\documentclass[12pt]{article}
\usepackage{amsmath}
\usepackage{amsthm}
\usepackage{amsfonts}
\usepackage[cmtip,all]{xy}
\usepackage{shuffle}
\newtheorem{lem}{Lemma}
\newtheorem{prop}{Proposition}
\newtheorem{thm}{Theorem}
\newtheorem{cor}{Corollary}
\begin{document}
\title{Quasi-shuffle algebras and applications}
\author{Michael E. Hoffman\\
\small U. S. Naval Academy\\[-0.8 ex]
\small Annapolis, MD 21402 USA\\[-0.8 ex]
\small\texttt{meh@usna.edu}}
\date{\small May 30, 2018
\small 2010 AMS Classification:  16T30, 11M32}
\maketitle
\def\sh{\shuffle}
\def\la{\lambda}
\def\ka{\kappa}
\def\ep{\epsilon}
\def\zt{\zeta}
\def\zts{\zeta^{\star}}
\def\zth{\zeta^{\frac12}}
\def\si{\sigma}
\def\om{\omega}
\def\De{\Delta}
\def\Ga{\Gamma}
\def\Si{\Sigma}
\def\Z{\mathbb Z}
\def\Q{\mathbb Q}
\def\R{\mathbb R}
\def\C{\mathbb C}
\def\P{\mathbb P}
\def\CC{\mathcal C}
\def\kA{k\langle A\rangle}
\def\op{\diamond}
\def\Ai{\operatorname{Ai}}
\def\Re{\operatorname{Re}}
\def\QS{\operatorname{QSym}}
\def\card{\operatorname{card}}
\def\E{\mathfrak E}
\def\ZZ{\mathfrak Z}
\def\I{\mathcal I}
\def\prd{\overset{r}*}
\def\phd{\,\overset{\scriptstyle{\frac12}}*}
\def\<{\langle}
\def\>{\rangle}
\begin{abstract}
Quasi-shuffle algebras have been a useful tool in studying 
multiple zeta values and related quantities, including multiple
polylogarithms, finite multiple harmonic sums, and $q$-multiple 
zeta values.  Here we show that two ideas previously considered 
only for multiple zeta values, the interpolated product of S. 
Yamamoto and the symmetric sum theorem, can be generalized to 
any quasi-shuffle algebra.
\end{abstract}
\section{Introduction}
Multiple zeta values and related quantities, although studied
by Euler in the simplest cases, only began to receive systematic
attention in the early 1990s.  Suddenly they seemed to be 
everywhere:  in high-energy physics, in knot theory, and in 
theoretical computer science.  Many early papers on these
quantities emphasized proofs of specific identities, and used
methods of analysis.  But from the beginning the importance 
of algebraic structure proved its importance.
\par
In \cite{H97} the author recognized multiple zeta values as homomorphic
images of quasi-symmetric functions, allowing the use of familiar results 
on symmetric functions in proving relations of multiple zeta values.
This was generalized in \cite{H00}, which introduced the quasi-shuffle 
product. (In the same year Li Guo and W. Keigher \cite{GK} independently 
introduced an essentially equivalent construction, but it took several 
years for the relation between the two to be generally recognized.)
\par
Meanwhile, the circle of ideas around multiple zeta values and
multiple polylogarithms continued to expand, and came to include
examples that went beyond the framework of \cite{H00}, particularly
$q$-multiple zeta values.
K. Ihara, J. Kajikawa, Y. Ohno and J. Okuda \cite{IKOO} generalized 
the definition of the quasi-shuffle product to include such cases,
but neglected much of the algebraic machinery developed in 
\cite{H00}, particularly the Hopf algebra structure and the linear
maps induced by formal power series.
In 2012 Ihara and the author set out to develop a generalization of the
definition used in \cite{H00} while retaining the algebraic
structures developed in that paper, and indeed extending them.
This led to \cite{HI}, which which presented such a generalization
and applied it to an array of examples.  We review this construction \S2.
\par
The methods introduced in \cite{HI} proved especially effective
in treating the interpolated multiple zeta values (or $r$-MZVs) 
introduced by S. Yamamoto \cite{Y}, which interpolate between ordinary
multiple zeta values ($r=0$) and multiple zeta-star values $r=1$).
Yamamoto showed that $r$-MZVs multiply according to an interpolated
product; in \S3 we define interpolated products on any quasi-shuffle algebra.
A quasi-shuffle algebra with the interpolated product has a Hopf algebra 
structure, generalizing the results of \cite{HI}.
\par
The algebraic machinery of \cite{HI}, which allows transparent proofs
of many results in \cite{IKZ} and \cite{IKOO}, is briefly introduced
in \S4 and applied to multiple zeta values in \S5.  We also give a
new result for multiple zeta-half values (i.e., $r$-MZVs with $r=\frac12$).
The same quasi-shuffle algebra that has the multiple zeta values as
homomorphic images also has as images various ``exotic'' multiple zeta values,
such as the multiple $t$-values \cite{H16}, the Bessel-function zeta 
values introduced by  T. V. Wakhare and C. Vignat \cite{WV1}, and the
Airy multiple zeta values, all discussed in \S6.
\par
In \S7 we consider a different quasi-shuffle algebra, which has as its
image the alternating or ``colored'' multiple zeta values.  Finally,
in \S8 we show how the symmetric sum theorems given in \cite{H92} for
multiple zeta values can be generalized to any quasi-shuffle algebra.
\section{The basic construction}
We begin by reviewing the construction given in \cite{HI}.
Let $A$ be a countable set $A$ of letters, $k$ a field.
We assume there is a commutative, associative product $\op$ on $kA$.
\par
Now let $\kA$ be the noncommutative polynomial algebra over $A$.
So $\kA$ is the vector space over $k$ generated by ``words'' 
(monomials) $a_1a_2\cdots a_n$, with $a_i\in A$:  for a word
$w=a_1\cdots a_n$ we write $\ell(w)=n$ (and we set $\ell(1)=0$).
Define a $k$-bilinear product $*$ on $\kA$ by making $1\in\kA$ 
the identity element for each 
product, and requiring that $*$ satisfy the relation
\begin{equation}
\label{recur}
(aw)*(bv)=a(w*bv)+b(aw*v)+(a\op b)(w*v)
\end{equation}
for all $a,b\in A$ and all monomials $w,v$ in $\kA$.
Then $(\kA,*)$ is a commutative algebra.
If the product $\op$ is identically zero, then $*$ 
coincides with the usual shuffle product $\sh$ on $\kA$.
We will need the following lemma in the next section.
\begin{lem} 
\label{bare}
For letters $a,b$ and words $v,w$ such that $v\ne 1\ne w$,
\begin{equation}
\label{one}
a\op(v*b)+ba\op v=(a\op v)*b+a\op bv,
\end{equation}
\begin{equation}
\label{io}
(a\op v)*(b\op w)=a\op (v*(b\op w))+b\op((a\op v)*w)-a\op b\op(v*w),
\end{equation}
and
\begin{multline}
\label{st}
a(v*(b\op w))+a\op(v*bw)+b((a\op v)*w)+b\op (av*w)=\\
av*(b\op w)+(a\op v)*bw+2 (a\op b)(v*w) .
\end{multline}
\end{lem}
\begin{proof}
Writing $v=cv'$ for a letter $c$, Eq. (\ref{one}) is
\[
a\op(cv'*b)+ba\op cv'=a\op cv'*b+a\op bcv',
\]
or
\[
a\op c(v'*b)+a\op bv+a\op b\op v+ba\op v=a\op c(v'*b)
+ba\op v+a\op b\op v+a\op bv,
\]
which is evidently true.  Setting also $w=dw'$, the left- and
right-hand sides of Eq. (\ref{io}) are
\[
a\op c(v'*(b\op w))+b\op d((a\op v)*w')+a\op b\op c\op d(v'*w')
\]
and 
\begin{multline*}
a\op c(v'*(b\op w))+a\op b\op d(v*w')+a\op c\op b\op d(v'*w')\\
+b\op a\op c(v'*w)+b\op d((a\op v)*w')+b\op a\op c\op d(v'*w')\\
-a\op b\op c(v'*w)-a\op b\op d(v*w')-a\op b\op c\op d(v'*w'),
\end{multline*}
respectively, and these agree after cancellation.
Using the same notation, we can rewrite the left-hand side of
Eq. (\ref{st}) as
\begin{multline*}
a(v*(b\op w))+b((a\op v)*w)+a\op(c(v'*bw)+b(v*w)+c\op b(v'*w))\\
+b\op(a(v*w)+d(av*w')+a\op d(v*w'))\\
=a(v*(b\op w))+b((a\op v)*w)+a\op c(v'*bw)+a\op b(v*w)+a\op c\op b(v'*w)\\
+b\op a(v*w)+b\op d(av*w')+b\op a\op d(v*w')
\end{multline*}
and the right-hand side of Eq. (\ref{st}) as
\begin{multline*}
a(v*(b\op w))+b\op d(av*w')+a\op b\op d(v*w')\\
+a\op c(v'*bw)+b((a\op v)*w)+a\op c\op b(v'*w)+2a\op b(v*w),
\end{multline*}
and these evidently agree.
\end{proof}
\par
If $\De$ denotes the usual deconcatenation on $\kA$, i.e.,
\begin{multline*}
\De(a_1a_2\cdots a_n)=1\otimes a_1a_2\cdots a_n+a_1\otimes a_2\cdots a_n
+\dots+a_1\cdots a_{n-1}\otimes a_n\\
+a_1a_2\cdots a_n\otimes 1,
\end{multline*}
then $(\kA,*,\De)$ is a Hopf algebra \cite[Thm. 4.2]{HI}.  
It is easy to see that it is a bialgebra, and (using the filtration 
of $\kA$ by word length) it is filtered connected; this makes existence 
of the antipode automatic.
\par
For a composition $I=(i_1,\dots,i_m)$ of $n$ and a word
$w=a_1\cdots a_n$ of $\kA$, define
\[
I[w]=(a_1\op\dots\op a_{i_1})(a_{i_1+1}\op\dots\op
a_{i_1+i_2})\cdots(a_{i_1+\dots+i_{m-1}+1}\op\dots\op a_n) .
\]
Let 
\[
f=c_1t+c_2t^2+c_3t^3+\cdots\in tk[[t]]
\]
be a formal power series.  We can define a $k$-linear map
$\Psi_f:\kA\to\kA$ by
\begin{equation}
\label{phif}
\Psi_f(w)=\sum_{I=(i_1,\dots,i_m)\in\CC(\ell(w))}c_{i_1}\cdots c_{i_m}I[w] ,
\end{equation}
where $\CC(n)$ is the set of compositions of $n$.
Then we have the following result.
\begin{thm}
\label{comp}
[{\cite[Thm.~3.1]{HI}}] For $f,g\in k[[t]]$ as specified above, 
$\Psi_f\Psi_g=\Psi_{f\circ g}$.
\end{thm}
Here are some examples.
First, it is immediate from equation (\ref{phif}) that $\Psi_t$ is the 
identity homomorphism of $\kA$.  
Also, $T=\Psi_{-t}$ sends a word $w$ to $(-1)^{\ell(w)}w$; evidently
$T$ is an involution.  We note that $\Si=\Psi_{\frac{t}{1-t}}$ and
$\Si^{-1}=\Psi_{\frac{t}{1+t}}$ are given by
\[
\Si(w)=\sum_{I\in\CC(\ell(w))}I[w] \quad\text{and}\quad
\Si^{-1}(w)=\sum_{I\in\CC(\ell(w))}(-1)^{\ell(w)-\ell(I)}I[w] ,
\]
where $\ell(I)$ is the number of parts of the composition $I$.
Evidently $\Si(aw)=a\Si(w)+a\op\Si(w)$ for letters $a$ and words $w$, 
and (as in \cite{IKOO}) this property can be used to define $\Si$.  
While $\Si$ and $T$ are not inverses, it is easy to see from
Theorem \ref{comp} that $T\Si T=\Si^{-1}$, from which it follows that
$\Si T$ and $T\Si$ are involutions.
\par
From \cite{H00} we have the (inverse) functions $\exp=\Psi_{e^t-1}$ and
$\log=\Psi_{\log(1+t)}$.  As shown in \cite[Theorem 2.5]{H00},
$\exp$ is an algebra isomorphism from $(\kA,\sh)$ to $(\kA,*)$.
We have the following identity.
\begin{thm} $\Si=\exp T\log T$ .
\end{thm}
\begin{proof} This follows from Theorem \ref{comp}, since 
$\exp T=\Psi_{e^{-t}-1}$, $\log T=\Psi_{\log(1-t)}$, and $\log(1-t)$ composed 
with $e^{-t}-1$ gives
\[
\frac1{1-t}-1=\frac{1-(1-t)}{1-t}=\frac{t}{1-t} .
\]
\end{proof}
\section{The interpolated product}
For any $r\in k$, define $\Si^r=\Psi_{\frac{t}{1-rt}}$;
it then follows immediately from Theorem \ref{comp} that $\Si^r\Si^s=
\Si^{r+s}$, and it is easily seen that
\[
\Si^r(aw)=a\Si^rw+ra\op\Si^rw
\]
for any letter $a$ and word $w$.
We now define the interpolated product $\prd$ by
\[
u\prd v=\Si^{-r}(\Si^ru*\Si^rv)
\]
for any words $u,v$.  
Henceforth we shall treat both concatenation and $\op$ has having
higher binding than $*$ and $\prd$, so the second identity of 
Lemma \ref{bare} reads
\[
a\op v*b\op w=a\op (v*b\op w)+b\op(a\op v*w)-a\op b\op(v*w) .
\]
\begin{lem}
\label{decor}
Lemma \ref{bare} remains true when $*$ is replaced by $\prd$.
\end{lem}
\begin{proof}
For each identity, first replace $v$ and $w$ by $\Si^rv$ and $\Si^rw$ 
respectively and then apply $\Si^{-r}$ to both sides.
After appropriate simplification and (in the case of identity 
(\ref{st})) cancellation, the conclusion follows.
\end{proof}
We now show that the product $\prd$ can be defined inductively 
by a rule similar to Eqn. (\ref{recur}) for the quasi-shuffle product $*$.
This rule was first given by Yamamoto \cite{Y} in the case of multiple
zeta values.
\begin{thm}
The product $\prd$ can be specified by setting $1\prd w=w\prd 1=w$
for any word $w$, $a\prd b=ab+ba+(1-2t)a\op b$ for any letters $a,b$, 
and
\begin{multline*}
av\prd bw = a(v\prd bw)+b(av\prd w)+(1-2r)a\op b(v\prd w)\\
+(r^2-r)a\op b\op (v\prd w)
\end{multline*}
for any letters $a,b$ and words $v,w$ such that $vw\ne 1$. 
\end{thm}
\begin{proof}
Evidently $1\prd w=w\prd 1$ for any word $w$, and for letters
$a,b$ we have
\begin{multline*}
a\prd b=\Si^{-r}(a*b)=\Si^{-r}(ab+ba+a\op b)=ab-ra\op b+ba-rb\op a+a\op b\\
=ab+ba+(1-2r)a\op b.
\end{multline*}
Now let $a,b$ be letters, $v\ne 1$ a word.  Then
\begin{multline*}
av\prd b=\Si^{-r}(\Si^r(av)*b)=\Si^{-r}(a\Si^rv*b+ra\op\Si^rv*b)=\\
\Si^{-r}(a(\Si^rv*b)+ba\Si^rv+a\op b\Si^rv+ra\op\Si^rv*b)=\\
a(v\prd b)-ra\op(v\prd b)+b\Si^{-r}a\Si^rv-rb\op\Si^{-r}a\Si^rv
+a\op bv-ra\op b\op v+ra\op v\prd b\\
=a(v\prd b)+bav+(1-r)a\op bv+(r^2-r)a\op b\op v
-r(a\op(v\prd b)+ba\op v-a\op v\prd b)\\
=a(v\prd b)+bav+(1-2r)a\op bv+(r^2-r)a\op b\op v,
\end{multline*}
where we used Lemma 2 in the last step.  Finally,
let $a,b$ be letters, $v,w$ words with $v\ne 1\ne w$.  Then
\begin{multline*}
av\prd bw=\Si^{-r}(\Si^rav*\Si^rbw)=
\Si^{-r}((a\Si^rv+ra\op\Si^rv)*(b\Si^rw+rb\op\Si^rw))\\
=\Si^{-r}(a\Si^rv*b\Si^rw+ra\Si^rv*b\op\Si^rw+ra\op\Si^rv*b\Si^rw
+r^2a\op\Si^rv*b\op\Si^rw)\\
=\Si^{-r}(a(\Si^rv*b\Si^rw)+b(a\Si^rv*\Si^rw)+(a\op b)(\Si^rv*\Si^rw)
+ra\Si^rv*b\op\Si^rw\\
+ra\op\Si^rv*b\Si^rw+r^2a\op\Si^rv*b\op\Si^rw\\
=a(v\prd\Si^{-r}b\Si^rw)-ra\op(v\prd\Si^{-r}b\Si^rw)
+b(\Si^{-r}a\Si^rv\prd w)-rb\op(\Si^{-r}a\Si^rv\prd w)\\
+(a\op b)(v\prd w)-ra\op b\op(v\prd w)+r\Si^{-r}a\Si^rv\prd b\op w
+ra\op v\prd\Si^{-r}b\Si^rw+r^2a\op v\prd b\op w\\
=a(v\prd bw)-ra(v\prd b\op w)-ra\op(v\prd bw)+r^2a\op(v\prd b\op w)
+b(av\prd w)-rb(a\op v\prd w)\\
-rb\op(av\prd w)+r^2b\op(a\op v\prd w)+(a\op b)(v\prd w)
-ra\op b\op(v\prd w)+rav\prd b\op w\\
-r^2a\op v\prd b\op w+ra\op v\prd bw-r^2a\op v\prd b\op w+r^2a\op v\prd b\op w\\
=a(v\prd bw)+b(av\prd w)+(a\op b)(v\prd w)
-ra\op b\op(v\prd w)-ra(v\prd b\op w)\\
-ra\op(v\prd bw)-rb(a\op v\prd w)-rb\op(av\prd w)+rav\prd b\op w
+ra\op v\prd bw\\
+r^2(a\op(v\prd b\op w)+b\op(a\op v\prd w)-a\op v\prd b\op w)\\
=a(v\prd bw)+b(av\prd w)+(1-2r)(a\op b)(v\prd w)+(r^2-r)a\op b\op(v\prd w),
\end{multline*}
where we used Lemma \ref{decor} in the last step.
\end{proof}
\par
If $r=1$, we write $\star$ instead of $\,\overset{\scriptstyle{1}}*$.
The product $\star$ has inductive rule
\[
av\star bw=a(v\star bw)+b(av\star w)-a\op b(v\star w),
\]
which is of the same form as Eq. (\ref{recur}).
As noted in \cite{HI}, $T:(\kA,\star)\to(\kA,*)$ and 
$T:(\kA,*)\to(\kA,\star)$ are isomorphisms.
These are special cases of the following result.
\begin{prop} $T:(\kA,\prd)\to(\kA,\overset{1-r}*)$ is an isomorphism.
\end{prop}
\begin{proof} 
First note that $\Si^s:(\kA,\prd)\to (\kA,\overset{r-s}*)$ 
is an isomorphism, and that $\Si^rT=T\Si^{-r}$ for all $r\in k$.  Then
\begin{multline*}
T(u\prd v)=T\Si^{-r}(\Si^ru*\Si^rv)=\Si^rT(\Si^ru*\Si^rv)=
\Si^r(T\Si^r\star T\Si^rv)=\\
\Si^r(\Si^{-r}Tu\star\Si^{-r}Tv)=Tu\overset{1-r}* Tv 
\end{multline*}
for $u,v\in\kA$, and the result follows.
\end{proof}
\par
In what follows, $R$ is the linear function on $\kA$ that reverses
words, i.e., $R(a_1a_2\cdots a_n)=a_n a_{n-1}\cdots a_1$.  We note that
$R$ commutes with $\Psi_f$ for all $f\in tk[[t]]$ \cite[Prop. 4.3]{HI}.
The following result generalizes \cite[Thm. 4.2]{HI}.
\begin{thm} $(\kA,\prd,\De)$ is a filtered connected Hopf algebra
with antipode $\Si^{1-2r}TR$.  Also, $\Si^r:(\kA,\prd,\De)\to(\kA,*,\De)$
is a Hopf algebra isomorphism.
\end{thm}
\begin{proof}
To see that $(\kA,\prd,\De)$ is a Hopf algebra, the main thing to
check is that $\De(w_1\prd w_2)=\De(w_1)\prd\De(w_2)$ for any two
words $w_1$ and $w_2$.  We do this inductively on the word length.
We can assume $w_1\ne 1\ne w_2$, so let $w_1=au$ and $w_2=bv$ for 
letters $a,b$.  Using Sweedler's notation
\[
\De(u)=\sum u_{(1)}\otimes u_{(2)},\quad \De(v)=\sum v_{(1)}\otimes v_{(2)},
\]
we have
\[
\De(au)=\sum au_{(1)}\otimes u_{(2)}+1\otimes au
\]
and
\[
\De(bv)=\sum bv_{(1)}\otimes v_{(2)}+1\otimes bv
\]
so that
\begin{multline*}
\De(w_1)\prd \De(w_2)=\sum(au_{(1)}\prd bv_{(1)})\otimes (u_{(2)}\prd v_{(2)})
+\sum au_{(1)}\otimes(u_{(2)}\prd bv)\\
+\sum bv_{(1)}\otimes (au\prd v_{(2)})+1\otimes (au\prd bv)\\
=\sum a(u_{(1)}\prd bv_{(1)})\otimes(u_{(2)}\prd v_{(2)})
+\sum b(au_{(1)}\prd v_{(1)})\otimes(u_{(2)}\prd v_{(2)})+\\
(1-2r)\sum a\op b(u_{(1)}\prd v_{(1)})\otimes(u_{(2)}\prd v_{(2)})
+(r^2-r)\sum a\op b\op (u_{(1)}\prd v_{(1)})\otimes(u_{(2)}\prd v_{(2)})\\
+\sum au_{(1)}\otimes(u_{(2)}\prd w_2)+\sum bv_{(1)}\otimes(w_1\prd v_{(2)})
+1\otimes a(u\prd w_2)+1\otimes b(w_1\prd v)\\
+(1-2r)1\otimes a\op b(u\prd v)+(r^2-r)1\otimes a\op b\op (u\prd v) .
\end{multline*}
Using the induction hypothesis, this is
\begin{multline*}
(a\otimes 1)(\De(u\prd w_2))+1\otimes a(u\prd w_2)
+(b\otimes 1)(\De(w_1\prd v))+1\otimes b (w_1\prd v)\\
+(1-2r)(a\op b\otimes 1)\De(u\prd v)+(1-2r)(1\otimes a\op b)\De(u\prd v)\\
+(r^2-r)(a\op b\otimes 1)\op\De(u\prd v)
+(r^2-r)(1\otimes a\op b)\op\De(u\prd v)
\end{multline*}
which can be recognized as
\[
\De(w_1\prd w_2)=\De(a(u\prd w_2)+b(w_1\prd v)+(1-2r)a\op b(u\prd v)
+(r^2-r)a\op b\op (u\prd v)) .
\]
\par
Now $\Si^r:(\kA,\prd)\to (\kA,*)$ is an algebra homomorphism by
definition, and is also a coalgebra map for $\De$ \cite[Thm. 4.1]{HI}.
Hence $\Si^r$ is a Hopf algebra isomorphism.  Also, if
\[
w=\sum_{w} w_{(1)}\otimes w_{(2)}
\]
for a nonempty word, then
\[
\Si^rw=\sum_{w} \Si^rw_{(1)}\otimes \Si^rw_{(2)}
\]
and we have
\[
\sum_w S_*\Si^r w_{(1)}*\Si^r w_{(2)}=0
\]
for $S_*=\Si TR$ the antipode of the Hopf algebra 
$(\kA,*,\De)$ \cite[Thm. 4.2]{HI}.
Apply $\Si^{-r}$ to get
\[
\sum_w \Si^{1-r}TR\Si^r w_{(1)}\prd w_{(2)}=0;
\]
but this shows that $\Si^{1-r}TR\Si^r=\Si^{1-2r}TR$ is the antipode
of $(\kA,\prd,\De)$.
\end{proof}
Of course if $r=0$ the Hopf algebra $(\kA,\prd,\De)$ is just
$(\kA,*,\De)$; the antipode is $\Si TR$.  
If $r=1$ we get $(\kA,\star,\De)$, and the antipode is $\Si^{-1}TR=T\Si R$.
For $r=\frac12$ the inductive rule for the product is
\[
av\phd bw=a(v\phd bw)+b(av\phd w)-\frac14 a\op b\op (v\phd w)
\]
and the antipode is simply $TR$.
\section{Algebraic formulas}
In \cite{IKZ} and \cite{IKOO} there are algebraic formulas 
involving $\exp$ and $\log$.  These can be proved systematically
from the following result of \cite{HI}, where for $w\in\kA$ and 
$f=c_1t+c_2t^2+\cdots\in tk[[t]]$, $f_\bullet(\la w)$ denotes
\[
\la c_1w+\la^2 c_2w\bullet w+\la^3c_3w\bullet w\bullet w+\cdots \in \kA[[\la]] 
\]
for $\bullet=*,\sh,\star,\op$.
\begin{thm} 
\label{gsf}
[{\cite[Thm.~5.1]{HI}}]
For any $f\in tk[[t]]$ and $w\in\kA$,
\[
\Psi_f\left(\frac1{1-\la w}\right)=\frac1{1-f_{\op}(\la w)} .
\]
\end{thm}
We write $\exp_\bullet(\la w)$ for $1+f_\bullet(\la w)$, $f=e^t-1$,
and $\log_\bullet(\la w)$ for $f_\bullet(\la w)$, $f=\log(1+t)$.
By applying Theorem \ref{gsf} with $f=\log(1-t)$, we get
\begin{equation}
\label{expg}
\exp_*(\log_{\op}(1+\la z))=\frac1{1-\la z} ,
\end{equation}
and by applying it with $f=e^t-1$ we get
\[
\exp_*(\la z)=\exp\left(\frac1{1-\la z}\right)=\frac1{2-\exp_{\op}(\la z)} .
\]
Another consequence of Theorem \ref{gsf} is the following.
\begin{cor}
\label{gsint}
[{\cite[Cor.~5.5]{HI}}]
For any $z\in\kA$ and $r\in k$,
\[
\Si^r\left(\frac1{1-\la z}\right)*\frac1{1-r\la z}=\frac1{1-(1-r)z} .
\]
\end{cor}
\section{Multiple zeta values}
For positive integers $i_1,\dots,i_k$ with $i_1>1$, the corresponding
multiple zeta value is defined by 
\[
\zt(i_1,\dots,i_k)=\sum_{n_1>n_2>\dots>n_k\ge 1}
\frac1{n_1^{i_1}n_2^{i_2}\cdots n_k^{i_k}} .
\]
Let $A=\{z_1,z_2,\dots\}$, with the operation $z_i\op z_j=z_{i+j}$.
The following result can be extracted from \cite{H97}.
\begin{thm} 
The Hopf algebra $(\Q\<A\>,*,\De)$ is isomorphic to the 
algebra $\QS$ of quasi-symmetric functions over $\Q$.
\end{thm}
Let $\Q\<A\>^0$ be the subspace of $\Q\<A\>$ generated by 1 and all words
that do not begin with $z_1$.  Then $(\Q\<A\>^0,*)$ is a subalgebra of
$(\Q\<A\>,*)$.  We write $\QS^0$ for the corresponding subalgebra of
$\QS$.  The following fact was proved in \cite{H97}.
\begin{thm}
The linear function $\zt:\QS^0\to\R$ defined by $\zt(z_{i_1}\cdots z_{i_k})=
\zt(i_1,\dots,i_k)$ is a homomorphism from $\QS^0$ to the reals
with their usual multiplication.
\end{thm}
If we take $z=z_k$ in Eq. (\ref{expg}) above, we get
\[
\sum_{n\ge 0}\la^n z_k^n=\exp_*(\log_{\op}(1+\la z_k))=
\exp_*\left(\sum_{j\ge 1}\frac{(-1)^{j-1}\la^jz_{kj}}{j}\right),
\]
or, after applying $\zt$,
\[
\sum_{n\ge 0}\la^n\zt(\{k\}_n)=\exp\left(
\sum_{j\ge 1}\frac{(-1)^{j-1}\la^j\zt(kj)}{j}\right) ,
\]
where $\{k\}_n$ means $k$ repeated $n$ times.  If $k=2$ the right-hand side
is
\[
\exp\left(\sum_{j\ge 1}\frac{B_{2j}(2\pi)^{2j}}{(2j)(2j)!}\la^j\right)=
\frac{\sinh(\pi\la)}{\pi\la} ,
\]
from which follows 
\begin{equation}
\label{repin1}
\zt(\{2\}_n)=\frac{\pi^{2n}}{(2n+1)!} ,
\end{equation}
and a similar argument gives
\begin{equation}
\label{repin2}
\zt(\{4\}_n)=\frac{2^{2n+1}\pi^{4n}}{(4n+2)!} .
\end{equation}
\par
Two remarkable results about multiple zeta values are (1) the ``sum theorem,''
i.e., the sum of all multiple zeta values of a fixed depth and weight $n$ 
is just $\zt(n)$, as in 
\[
\zt(4,1,1)+\zt(3,2,1)+\zt(3,1,2)+\zt(2,3,1)+\zt(2,2,2)+\zt(2,1,3)=\zt(6),
\]
and, (2) the ``duality theorem,'' i.e., there is an involution 
$\tau:\QS^0\to\QS^0$ so that $\zt(\tau(u))=\zt(u)$, as in 
$\zt(3,1,2)=\zt(2,3,1)$.
To describe $\tau$ in terms of our algebraic setup, introduce two 
noncommuting variables $x$ and $y$, and set $z_i=x^{i-1}y$.  Then
$\QS^0$ is just the subspace of $\Q\<x,y\>$ generated by 1 and words that
begin with $x$ and end with $y$:  the function $\tau$ is the anti-isomorphism
exchanging $x$ and $y$ (so, e.g., $\tau(z_3z_1z_2)=\tau(x^2y^2xy)=xyx^2y^2=
z_2z_3z_1$).
\par
If we let $\zt^r=\zt\circ\Si^r$, then $\zt^r(w)$ is exactly the interpolated
multiple zeta value as defined by Yamamoto \cite{Y}.
Thus $\zt^0(w)=\zt(w)$ and $\zt^1(w)=\zt^{\star}(w)$ is the multiple zeta-star 
value defined by
\[
\zt^{\star}(i_1,\dots,i_k)=\sum_{n_1\ge n_2\ge\dots\ge n_k\ge 1}\frac1{n_1^{i_1}n_2^{i_2}
\cdots n_k^{i_k}} .
\]
Yamamoto showed that the interpolated multiple zeta values satisfy the following
version of the sum theorem, which is proved another way in \cite{HI}.
\begin{thm} If $n\ge 2$, then
\[
\sum_{\substack{i_1+\dots+i_l=n\\ i_1>1}}\zt^r(i_1,\dots,i_k)=
\zt(n)\sum_{k=0}^{l-1}r^n\binom{n-l-1+k}{k} .
\]
\end{thm}
Formulas for repeated values $\zt^r(\{m\}_n)$ can be obtained from
those for $\zt(\{m\}_n)$:  from Corollary \ref{gsint} it follows
that if 
\[
Z(\la)=\sum_{n=0}^\infty \zt(\{m\}_n)\la^n,
\]
then
\[
\sum_{n=0}^\infty \zt^r(\{m\}_n)\la^n=\frac{Z((1-r)\la)}{Z(-r\la)} .
\]
Hence, e.g.,
\[
\sum_{n=0}^\infty \zt^r(\{2\}_n)\la^n=\sqrt{\frac{r}{1-r}}
\frac{\sinh(\pi\sqrt{(1-r)\la})}{\sin(\pi\sqrt{r\la})} .
\]
\par
For interpolated multiple zeta values $\zt^r$ with $r=\frac12$
there is a ``totally odd sum theorem.''  This follows from two
known results:  the cyclic sum theorem and the two-one theorem.
Define the cyclic sum operation on $\QS^0$ by
\begin{multline*}
C(x^{i_1-1}yx^{i_2-1}y\cdots x^{i_k-1}y)=x^{i_1}yx^{i_2-1}y\cdots x^{i_k-1}y\\
+x^{i_2}yx^{i_3-1}y\cdots x^{i_k-1}yx^{i_1-1}y
+\dots+x^{i_k}yx^{i_1-1}y\cdots x^{i_{k-1}-1}y .
\end{multline*}
Then the cyclic sum theorem for multiple zeta-star values \cite{OW} 
asserts that 
\[
\zts(\tau C(w))=(n-1)\zt(n)
\]
for any word $w\in\QS^0$ of degree $n-1$.  The two-one formula 
\cite{Y,Z} gives
\[
\zts((xy)^{j_1}y(xy)^{j_2}y\cdots (xy)^{j_l}y)
=2^l\zth(x^{2j_1}yx^{2j_2}y\cdots x^{2j_l}y)
\]
for any sequence $(j_1,\dots,j_l)$ of nonnegative integers with $j_1>0$.
\begin{thm}
\label{tost}
Let $n>2$, $l<n$ be positive integers of the same parity.  Then
\[
\sum_{\substack{a_1+\dots+a_l=n\\ \text{$a_i$ odd},\ a_1>1}}\zth(a_1,\dots,a_l)
=\frac{n-1}{n-l}\binom{\frac{n+l}{2}-2}{l-1}\frac{\zt(n)}{2^{l-1}} 
=\frac{n-1}{\frac{n+l}2-1}\binom{\frac{n+l}2-1}{l-1}\frac{\zt(n)}{2^l} .
\]
\end{thm}
\begin{proof}
By the two-one formula
\[
\sum_{\substack{a_1+\dots+a_l=n\\ \text{$a_i$ odd},\ a_1>1}}\zth(a_1,\dots,a_l)=
2^{-l}\sum_{\substack{j_1+\dots+j_l=\frac{n-l}2\\ j_i\ge 0,\ j_1\ge 1}}
\zts((xy)^{j_1}y(xy)^{j_2}y\cdots (xy)^{j_l}y) .
\]
The latter sum has
\begin{equation}
\label{bin}
\binom{\frac{n+l}2-2}{l-1}
\end{equation}
terms.  To see this, note that written in the sequence notation each
term corresponds to a string
\begin{equation}
\label{string}
\underbrace{2,\dots,2}_{j_1},1,\underbrace{2,\dots,2}_{j_2},1,\dots,
\underbrace{2,\dots,2}_{j_l},1 
\end{equation}
with $j_i\ge 0$, $j_1\ge 1$, and $\sum_{i=1}^lj_i=\frac{n-l}2$.
Now the string (\ref{string}) always starts with 2 and ends with 1,
so we can think about the middle part:  it has length $\frac{n+l}2-2$, and
consists of $\frac{n-l}2-1$ twos and $l-1$ ones.
To specify such a string, we need only give the $l-1$ positions where the
ones go; so such strings are counted by the binomial coefficient (\ref{bin}).
\par
Now each word $u$ of the form
\begin{equation}
\label{word1}
(xy)^{j_1}y(xy)^{j_2}y\cdots (xy)^{j_l}y 
\end{equation}
with $\sum_{i=1}^l j_i=\frac{n-l}2$ and $j_1>0$ has
\[
\tau(u)=x^{i_1-1}yx^{i_2-1}y\cdots x^{i_k-1}y
\]
with $i_1>2$, $i_2,\dots,i_k>1$, $i_1+\dots+i_k=n$, and $k=\frac{n-l}2$.
These are exactly the words that appear in $C(w)$ for $w$ of the
form $x^{a_1-1}yx^{a_2-1}y\cdots x^{a_k-1}y$ with $a_1,\dots a_2>1$,
$a_1+\dots+a_k=n-1$, and $k=\frac{n-1}2$.
For any such $w$ the expansion of $\tau C(w)$ will have $\frac{n-l}2$ 
terms, so each term $\zts(u)$ contributes
\[
\frac2{n-l}(n-1)\zt(n),
\]
and the result follows.
(It may happen that $\frac2{n-l}\binom{\frac{n+l}2-2}{l-1}$ is not 
an integer, but the preceding sentence is still true since in
that case there are duplications in one or more of the images 
under $\tau C$.)
\end{proof}
\par
From the definition of the zeta-half values we get the following
corollary of Theorem \ref{tost}.
\begin{cor}
\label{red}
The sum 
\[
\sum_{\substack{a_1+\dots+a_l=n\\ \text{$a_i$ odd},\ a_1>1}}\zt(a_1,\dots,a_l)
\]
is a rational linear combination of multiple zeta values
of weight $n$ and depth less than $l$.
\end{cor}
In the depth three case we can say more.
\begin{cor}
If $n$ is odd, the sum
\[
\sum_{\substack{a_1+a_2+a_3=n\\ \text{$a_i$ odd},\ a_1>1}}\zt(a_1,a_2,a_3)
\]
is a polynomial in the ordinary zeta values with rational coefficients.
\end{cor}
\begin{proof}
By Corollary \ref{red}, the sum can be written as a rational linear
combination of single and double zeta values of weight $n$.  But double 
zeta values of odd weight are known to be rational polynomials in the 
ordinary zeta values, and the conclusion follows.
\end{proof}
\section{``Exotic'' multiple zeta values}
In this section we give some examples of ``exotic'' homomorphic
images of subalgebras of $\QS$.
Our first example involves the multiple $t$-values as defined in \cite{H16}.
For positive integers $i_1,\dots,i_k$ with $i_1>1$,
let
\[
t(i_1,\dots,i_k)=\sum_{\substack{n_1>n_2>\dots>n_k\ge 1\\ \text{$n_j$ odd}}}
\frac1{n_1^{i_1}n_2^{i_2}\cdots n_k^{i_k}} .
\]
Then $t:\QS^0\to\R$ defined by $t(z_{i_1}\cdots z_{i_k})=t(i_1,\dots,i_k)$
defines a homomorphism.
The multiple $t$-values have obvious parallels with multiple zeta values;
for example, it is evident that $t(n)=(1-2^{-n})\zt(n)$ for $n\ge 2$.
Also, paralleling the identities (\ref{repin1}) and (\ref{repin2}) of the 
last section we have from \cite{H16}
\begin{equation}
\label{repan}
t(\{2\}_n)=\frac{\pi^{2n}}{2^{2n}(2n)!},\quad
t(\{4\}_n)=\frac{\pi^{4n}}{2^{2n}(4n)!} .
\end{equation}
\par
Following Wakhare and Vignat \cite{WV1}, we can take any function $G$
with real zeros $\{a_1,a_2,\dots\}$ such that $\lim_{n\to\infty} |a_n|=\infty$,
and define a homomorphism $\zt_G:S\to\R$ by sending
$z_{i_1}\cdots z_{i_l}$ to 
\[
\zt_G(i_1,\dots,i_l)=
\sum_{n_1>n_2>\dots>n_k\ge 1}\frac1{a_{n_1}^{i_1}a_{n_2}^{i_2}\cdots a_{n_l}^{i_l}}
\]
for some subalgebra $S$ of $\QS$ that depends on the growth rate of 
$|a_n|$ with $n$.  Wakhare and Vignat consider the case where
$a_n$ is the $n$th positive zero of the Bessel function $J_{\nu}$
of the first kind of order $\nu$.  They obtain the remarkable
formulas
\begin{align}
\label{bes2}
\zt_{J_{\nu}}(\{2\}_n)&=\frac1{2^{2n}n!(\nu+1)(\nu+2)\cdots (\nu+n)},\\
\label{bes4}
\zt_{J_{\nu}}(\{4\}_n)&=\frac1{2^{4n}n!(\nu+1)\cdots (\nu+2n)(\nu+1)\cdots 
(\nu+n)}.
\end{align}
We note that since
\[
J_{\frac12}(z)=\sqrt{\frac2{\pi z}}\sin z\quad\text{and}\quad
J_{-\frac12}(z)=\sqrt{\frac2{\pi z}}\cos z 
\]
we have
\[
\pi^{|w|}\zt_{J_{\frac12}}(w)=\zt(w)\quad\text{and}\quad
\left(\frac{\pi}2\right)^{|w|}\zt_{J_{-\frac12}}(w)=t(w),
\]
and thus Eqs. (\ref{bes2}) and (\ref{bes4}) imply Eqs. 
(\ref{repin1}), (\ref{repin2}), and (\ref{repan}) above.
\par
We can also choose $0>a_1>a_2>\cdots $ to be the zeros
of the Airy function $\Ai(z)$.  
Now $\Ai(z)$ has the infinite product expansion \cite[p. 18]{VS}
\begin{equation}
\label{aprod}
\Ai(z)=\Ai(0)e^{-\kappa z}\prod_{n=1}^\infty
\left(1-\frac{z}{a_n}\right)e^{\frac{z}{a_n}} ,
\end{equation}
where
\[
\ka=\left|\frac{\Ai'(0)}{\Ai(0)}\right|=\frac{3^{\frac56}\Ga(\tfrac23)^2}
{2\pi} \approx 0.729011 .
\]
Starting with Eq. (\ref{aprod}), take logarithms and differentiate to
get
\[
\frac{d}{dz}\log\Ai(z)=-\ka+\sum_{n=1}^\infty
\left[\frac1{a_n}+\frac1{z-a_n}\right] .
\]
Then evidently
\begin{equation}
\label{sern}
\frac{d^k}{dz^k}\log\Ai(z)=\sum_{n=1}^\infty\frac{(-1)^{k-1}(k-1)!}{(z-a_n)^k} 
\end{equation}
for $k\ge 2$.
Since $\Ai''(z)=z\Ai(z)$, we have
\begin{equation}
\label{d2}
\frac{d^2}{dz^2}\log\Ai(z)=z-\frac{\Ai'(z)^2}{\Ai(z)^2} .
\end{equation}
Combining Eq. (\ref{sern}) for $k=2$ and Eq. (\ref{d2}), we have
\begin{equation}
\label{d2a}
\sum_{n=1}^\infty\frac{-1}{(z-a_n)^2}=z-\frac{\Ai'(z)^2}{\Ai(z)^2} ,
\end{equation}
which at $z=0$ gives
\begin{equation}
\label{two}
\zt_{\Ai}(2)=\sum_{n=1}^\infty \frac1{a_n^2}=\ka^2 .
\end{equation}
Repeated differentiation of $f(z)=\Ai'(z)/\Ai(z)$ gives the following
result, originally due to Crandall \cite{C}.
\begin{thm} For all $n\ge 2$, $\zt_{\Ai}(n)$ is a rational polynomial
in $\ka$ of degree $n$, with leading coefficient 1.
\end{thm}
Also, from Eq. (\ref{aprod}) it follows that
\[
\Ai(z)\Ai(-z)=\Ai(0)^2\prod_{k=1}^\infty\left(1-\frac{z^2}{a_k^2}\right)
\]
and thus that
\begin{multline*}
\sum_{n=0}^\infty \zt_{\Ai}(\{2\}_n)(-1)^nz^{2n}=\frac{\Ai(z)\Ai(-z)}{\Ai(0)^2}=\\
1-\ka^2z^2+\frac{\ka}{6}z^4-\frac1{60}z^6+\frac{\ka^2}{336}z^8-\frac{\ka}{6480}
z^{10}+\cdots .
\end{multline*}
By comparison with the series \cite{R}
\[
\Ai(z)\Ai(-z)=\frac2{\sqrt{\pi}}\sum_{n\ge 0}\frac{(-1)^nz^{2n}}{12^{\frac{2n+5}{6}}
n!\Ga(\frac{2n+5}6)}
\]
it can be seen that $\zt_{\Ai}(\{2\}_n)$ is rational if $n\equiv 0$ 
mod 3, a rational multiple of $\ka^2$ if $n\equiv 1$ mod 3, and a rational
multiple of $\ka$ if $n\equiv 2$ mod 3.
Further formulas for $\zt_{\Ai}(\{2\}_n)$ and also for $\zt_{\Ai}(\{4\}_n)$
were given by Wakhare and Vignat \cite{WV2}.
\section{Alternating multiple zeta values}
Let $r$ be a positive integer, 
$A=\{z_{m,j} |\ m\in\Z^+, i\in\{0,1,\dots,r-1\}\}$, with $z_{m,j}\op z_{n,k}=
z_{m+n,j+k}$, where addition in the second subscript is understood mod $r$.
Then $(\Q\<A\>,*)$ is the ``Euler algebra'' $\E_r$ as defined in \cite{H00}.
If we let $\E_r^0$ be the subalgebra generated by 1 and all words that
do not begin with $z_{1,0}$, then there is a homomorphism $\ZZ_r:\E_r^0\to\C$
sending $z_{m_1,j_1}\cdots z_{m_k,j_k}$ to 
\[
\sum_{n_1>\dots>n_k\ge 1}\frac{\ep^{n_1j_1}\cdots\ep^{n_kj_k}}{n_1^{m_1}\cdots n_k^{m_k}},
\]
where $\ep=e^{\frac{2\pi i}{r}}$.  Of course $\E_1$ is just $\QS$,
with $\ZZ_1=\zt$.  In the case $r=2$ the image of $\ZZ_r$
is real-valued, and $\ZZ_r$ sends a monomial to what is usually called an
alternating or ``colored'' multiple zeta value.  In this case we can
adapt the sequence notation of multiple zeta values and write, e.g.,
$\zt(\bar1,2,\bar3)$ for $\ZZ_2(z_{1,1}z_{2,0}z_{3,1})$.
Evidently $\zt(\bar1)=-\log2$ and $\zt(\bar k)=(2^{-k+1}-1)\zt(k)$ for
$k\ge 2$.  Generating functions for $\zt(\{\bar k\}_n)$ are discussed
already in \cite{BBB}.  A notable case is
\begin{equation}
\label{b1}
\sum_{n=0}^\infty \zt(\{\bar 1\}_n)\la^n=\frac{\sqrt{\pi}}{\Ga(\frac{1-\la}2)
\Ga(1+\frac{\la}2)} .
\end{equation}
The theory of interpolated products carries over to this case; for
example
\[
\zt^r(\bar1,2,\bar3)=\zt(\bar1,3,\bar3)+r\zt(\bar3,\bar3)+r\zt(\bar1,\bar5)
+r^2\zt(6) .
\]
We can generalize formulas like (\ref{b1}) to interpolated alternating
multiple zeta values:
\[
\sum_{n=0}^\infty \zt^r(\{\bar 1\}_n)\la^n=\frac{\Ga(\frac{1+r\la}2)
\Ga(1-\frac{r\la}2)}{\Ga(\frac{1-(1-r)\la}2)\Ga(1+\frac{(1-r)\la}2)} .
\]
Some results for alternating multiple zeta values can be stated
in terms of interpolated values, such as the following one of C.
Glanois \cite{G}.
\begin{thm} If $s_1,\dots, s_r$ is a sequence of elements of 
$\{1,\bar2,3,\bar4,5,\dots\}$ with $s_1\ne 1$,
then the interpolated alternating multiple zeta value 
$\zt^{\frac12}(s_1,\dots,s_r)$ is a rational linear combination of 
multiple zeta values.
\end{thm}
\section{Symmetric sum theorems}
The prototypical symmetric sum theorem was proved in \cite{H92}.
\begin{thm}
\label{proto}
[{\cite[Thm.~2.2]{H92}}]
If $k_1,\dots,k_n\ge 2$, then
\[
\sum_{\si\in S_n}\zt(k_{\si(1)},\dots,k_{\si(n)})=
\sum_{B=\{B_1,\dots,B_l\}\in\Pi_n}c(B)\prod_{m=1}^l\zt\left(\sum_{j\in B_m}k_j\right)
\]
where $S_n$ is the symmetric group on $n$ letters, 
$\Pi_n$ is the set of partitions of the set $\{1,\dots,n\}$, and 
\[
c(B)=(-1)^{k-l}(\card B_1-1)!(\card B_2-1)!\cdots (\card B_l-1)!
\]
for $B=\{B_1,\dots,B_l\}\in\Pi_n$.
\end{thm}
In fact, as noted in \cite{H15}, this identity can be proved in 
$\QS$ by M\"obius inversion and then (if all the $k_i\ge 2$) transferred 
to the reals via the homomorphism $\zt:\QS^0\to\R$.
But in fact it can be generalized in two ways:  first, it is true
for {\it any} quasi-shuffle algebra $(\Q\<A\>,*)$, and second, we
can extend it to the interpolated product.  The result is as follows.
\begin{thm} 
\label{syform}
If $u_1,\dots,u_n\in A$, then in $(\Q\<A\>,\prd)$
\begin{equation}
\label{form}
\sum_{\si\in S_k} u_{\si(1)}u_{\si(2)}\cdots u_{\si(k)}=
\sum_{B=\{B_1,\dots,B_l\}\in\Pi_k}c_r(B)u_{B_1}\prd u_{B_2}\prd\cdots\prd u_{B_l} ,
\end{equation}
where $u_{B_i}=\op_{j\in B_i}u_j$, 
$p_a(r)=(1-r)^a-(-r)^a$, 
and 
\[
c_r(B)=(-1)^{k-l}\prod_{m=1}^l(\card B_m-1)!p_{\card B_m}(r) 
\]
for $B=\{B_1,\dots,B_l\}\in\Pi_k$.
\end{thm}
\begin{proof}
We write $S(a,b)=ab+ba$, $S(a,b,c)=abc+acb+bac+bca+cab+cba$, 
and so on, so Eq. (\ref{form}) is
\[
S(u_1,\dots,u_n)=\sum_{\substack{\text{partitions $\Pi=(P_1,\dots, P_l)$}\\ 
\text{of $\{1,\dots,n\}$}}}
c_r(\Pi)u_{P_1}\prd u_{P_2}\prd\cdots\prd u_{P_l} .
\]
We proceed by induction on $n$.  Take the $\prd$-product of both 
sides of Eq. (\ref{form}) with $u_{n+1}$ to get
\begin{multline*}
S(u_1,\dots,u_{n+1})+(1-2r)[S(u_1\op u_{n+1},u_2,\dots,u_n)
+S(u_1,u_2\op u_{n+1},\dots,u_n)\\
+\dots+S(u_1,\dots,u_{n-1},u_n\op u_{n+1})]
+2(r^2-r)[S(u_1\op u_2\op u_{n+1},u_3,\dots,u_n)\\
+S(u_1\op u_3\op u_{n+1},u_2,u_4,\dots,u_n)+\dots+
S(u_{n-1}\op u_n\op u_{n+1},u_1,\dots,u_{n-2})]\\
=\sum_{\substack{\text{partitions $\Pi=(P_1,\dots, P_l)$}\\ 
\text{of $\{1,\dots,n\}$}}}
c_r(\Pi)u_{P_1}\prd u_{P_2}\prd\cdots\prd u_{P_l}\prd u_{n+1}
\end{multline*}
or
\begin{multline}
\label{mess}
S(u_1,\dots,u_{n+1})=-(1-2r)[S(u_1\op u_{n+1},u_2,\dots,u_n)
+S(u_1,u_2\op u_{n+1},\dots,u_n)\\
+\dots+S(u_1,\dots,u_{n-1},u_n\op u_{n+1})]
-2(r^2-r)[S(u_1\op u_2\op u_{n+1},u_3,\dots,u_n)\\
+S(u_1\op u_3\op u_{n+1},u_2,
u_4,\dots,u_n)+\dots+S(u_{n-1}\op u_n\op u_{n+1},u_1,\dots,u_{n-2})]\\
+\sum_{\substack{\text{partitions $\Pi=(P_1,\dots, P_l)$}\\ 
\text{of $\{1,\dots,n+1\}$ having}\\ \text{$\{n+1\}$ as a part}}}
c_r(\Pi)u_{P_1}\prd u_{P_2}\prd\cdots\prd u_{P_l} .
\end{multline}
We must show that the right-hand side of this equation coincides with
\begin{equation}
\label{exp}
\sum_{\substack{\text{partitions $\Pi=(P_1,\dots, P_l)$}\\ 
\text{of $\{1,\dots,n+1\}$}}}
c_r(\Pi)u_{P_1}\prd u_{P_2}\prd\cdots\prd u_{P_l} ,
\end{equation}
which we shall do by considering whether the cardinality of the part 
of $\Pi$ to which $n+1$ belongs is 1, 2, or $\ge 3$.
\par
Note that there are three groups on terms on the right-hand side of 
Eq. (\ref{mess}).
If $\{n+1\}$ is a part of $\Pi$, the corresponding term in (\ref{exp})
is contributed by the third group of terms on the right-hand side
of (\ref{mess}).
\par
Suppose now that $n+1$ belongs to a part of cardinality 2 in 
$\Pi=(P_1,\dots,P_l)$, say $P_1$.
The term corresponding to $\Pi$ in (\ref{exp}) only arises 
(via the induction hypothesis) from the first group of terms
on the right-hand side of (\ref{mess}), and the coefficient of 
$u_{P_1}\cdots u_{P_l}$ is
\begin{multline*}
-(1-2r)(-1)^{n-l}(\card P_2-1)!\cdots (\card P_l-1)!
p_{\card P_2}(t)\cdots p_{\card P_l}(t)\\
=(-1)^{n+1-l}(\card P_1-1)!\cdots (\card P_l-1)!
p_{\card P_1}(t)\cdots p_{\card P_l}(t) .
\end{multline*}
\par
Finally, suppose $n+1$ belongs to a part $P_1$ of $\Pi$ with
cardinality $k\ge 3$.  The term $u_{P_1}\cdots u_{P_l}$ arises from
the first group of terms in $k-1$ ways, contributing coefficient
\[
-(k-1)(1-2r)(-1)^{n-l}p_{k-1}(r)(k-2)!C,
\]
where
\[
C=(\card P_2-1)!\cdots (\card P_l-1)!p_{\card P_2}(t)\cdots p_{\card P_l}(t) .
\]
The same term arises from the second group of terms in $\binom{k-1}{2}$
ways, contributing coefficient
\[
-\binom{k-1}{2}2(r^2-r)(-1)^{n-1-l}p_{k-2}(r)(k-3)!C,
\]
and it suffices to show
\[
(1-2r)p_{k-1}(r)-(r^2-r)p_{k-2}(r)=p_k(r) ,
\]
which is immediate.
\end{proof}
Note that $p_a(0)=1$ and $p_a(1)=(-1)^{a-1}$, so $c_0(\Pi)=c(\Pi)$
and $c_1(\Pi)=|c(\Pi)|$, making Theorem \ref{syform}
reduce to 
\[
\sum_{\si\in S_n}u_{\si(1)}u_{\si(2)}\cdots u_{\si(n)}=
\sum_{\substack{\text{partitions $\Pi=(P_1,\dots,P_l)$}\\ \text{of $\{1,\dots,n\}$}}}
c(\Pi)u_{P_1}* u_{P_2}*\cdots* u_{P_l} 
\]
in the case $r=0$; if $r=1$ we get
\[
\sum_{\si\in S_n}u_{\si(1)}u_{\si(2)}\cdots u_{\si(n)}=
\sum_{\substack{\text{partitions $\Pi=(P_1,\dots,P_l)$}\\ \text{of $\{1,\dots,n\}$}}}
|c(\Pi)|u_{P_1}\star u_{P_2}\star\cdots\star u_{P_l}  .
\]
Also,
\[
p_a\left(\frac12\right)=\begin{cases} 0,&\text{if $a$ even,}\\
2^{1-a},&\text{if $a$ odd,}\end{cases}
\]
so that only partitions with all parts of odd cardinality appear when 
$r=\frac12$.  In fact
\[
c_{\frac12}(\Pi)=\begin{cases} \left(\frac12\right)^{n-l}\prod_{i=1}^l
(\card P_i-1)!,&\text{if $\card P_1\cdots \card P_l$ is odd;}\\
0,&\text{otherwise.}\end{cases}
\]
\par
If in Theorem \ref{syform} we take $A=\{z_1,z_2,\dots\}$ with 
$z_i\op z_j=z_{i+j}$ and $u_i=z_{k_i}$, $1\le i\le n$ 
(with $k_i\ne 1$ for all $i$), we get
\begin{equation}
\label{syzt}
\sum_{\si\in S_n}\zt^r(k_{\si(1)},\dots,k_{\si(n)})=
\sum_{\substack{\text{partitions $\Pi=(P_1,\dots,P_l)$}\\ \text{of $\{1,\dots,n\}$}}}
c_r(\Pi)\prod_{j=1}^l\zt\left(\sum_{h\in P_j}k_h\right) ,
\end{equation}
generalizing Theorem \ref{proto}; in fact $r=0$ gives Theorem \ref{proto}
and $r=1$ gives the corresponding result for star-zeta values 
\cite[Thm. 2.1]{H92}.
Identity (\ref{syzt}) holds with $t$ (or $\zt_{J_{\nu}}$ or $\zt_{\Ai}$)
in place of $\zt$.
\par
From Theorem \ref{syform} we can obtain a result in terms of integer 
partitions.
\begin{cor}
\label{repmzv}
If $u\in A$, then in $(\Q\<A\>,\prd)$
\[
u^n=\sum_{\la\vdash n}\frac{\ep_{\la}}{z_{\la}}\prod_{j=1}^{\ell(\la)}
p_{\la_j}(r)u^{\op\la_1}\prd\cdots\prd u^{\op\la_l}
\]
where $u^{\op n}$ means $\underbrace{u\op\cdots\op u}_n$ and (as in \cite{M}) 
$\ep_{\la}=(-1)^{n-\ell(\la)}$ and 
$z_{\la}=m_1(\la)!1^{m_1(\la)}m_2(\la)!2^{m_2(\la)}\cdots$, for 
$m_i(\la)$ the multiplicity of $i$ in $\la$.
\end{cor}
\begin{proof}
Set $u_1=\dots=u_n=u$ in Theorem \ref{syform} to get
\[
n!u^n=
\sum_{\substack{\text{partitions}\\ \Pi=(P_1,\dots,P_l)\\ \text{of $\{1,\dots,n\}$}}}(-1)^{n-l}(\la_1-1)!
\cdots (\la_l-1)!p_{\la_1}(r)\cdots p_{\la_l}(r)u_{i\la_1}\prd\cdots\prd u_{i\la_l},
\]
where we write $\la_i=\card P_i$.  Now the number of set partitions
$(P_1,\dots,P_l)$ of $\{1,\dots,n\}$ corresponding to the integer partition
$\la=(\la_1,\dots,\la_l)$ of $n$ is
\[
\frac1{m_1(\la)!m_2(\la)!\cdots}\binom{n}{\la_1}\binom{n-\la_1}{\la_2}\cdots
=\frac1{m_1(\la)!m_2(\la)!\cdots}\frac{n!}{\la_1!\la_2!\cdots \la_l!} .
\]
Thus $u^n$ is
\begin{multline*}
\sum_{\substack{\text{partitions}\\ \la=(\la_1,\dots,\la_l)\\ \text{of $n$}}}
\frac{(-1)^{n-l}(\la_1-1)!\cdots(\la_l-1)!}{m_1(\la)!m_2(\la)!\cdots 
\la_1!\cdots\la_l!}
p_{\la_1}(r)\cdots p_{\la_l}(r) u^{\op\la_1}\prd\cdots\prd u^{\op\la_l}\\
=\sum_{\substack{\text{partitions}\\ \la=(\la_1,\dots,\la_l)\\ \text{of $n$}}}
\frac{\ep_{\la}}{z_{\la}}p_{\la_1}(r)\cdots p_{\la_l}(r)
u^{\op\la_1}\prd\cdots\prd u^{\op\la_l} .
\end{multline*}
\end{proof}
Applying $\zt^r$ to the corollary with $u=z_i$, $i\ge 2$, we obtain
\[
\zt(z_i^n)=\sum_{\la\vdash n}\frac{\ep_{\la}}{z_{\la}}\prod_{j=1}^{\ell(\la)}
p_{\la_j}(r)\zt(i\la_j) .
\]
In the cases $r=0,1,\frac12$, this identity is respectively
\begin{align}
\label{elem}
\zt(\{i\}_n)&=\sum_{\la\vdash n}\frac{\ep_{\la}}{z_{\la}}\prod_{j=1}^{\ell(\la)}
\zt(i\la_j)\\
\label{compl}
\zts(\{i\}_n)&=\sum_{\la\vdash n}\frac1{z_{\la}}\prod_{j=1}^{\ell(\la)}\zt(i\la_j)\\
\label{half}
\zth(\{i\}_n)&=\sum_{\substack{\la\vdash n\\ \text{all parts of $\la$ odd}}}
\frac1{2^{n-\ell(\la)}z_{\la}}\prod_{j=1}^{\ell(\la)}\zt(i\la_j) .
\end{align}
Eqs. (\ref{compl}) and (\ref{elem}) are homomorphic images of the two
parts of \cite[Eq. ($2.14^{\prime}$)]{M}.
Eq. (\ref{half}) is obtained a different way in \cite{HI}
(see Eq. (41)).
\par
We note that Eq. (\ref{syzt}) applies to alternating multiple zeta values
as well, provided we define addition on the set 
$\I=\{\dots,\bar2,\bar1,1,2,\dots\}$ of indices 
to agree with usual addition on $\{1,2,\dots,\}$ and extend it to $\I$ via
\begin{align*}
a+\bar b=\bar a+b&=\overline{a+b}\\
\bar a+\bar b&=a+b
\end{align*}
for positive integers $a,b$.  Thus, e.g.,
\begin{multline*}
\zt^r(\bar1,2,\bar3)+\zt^r(\bar1,\bar3,2)+\zt^r(2,\bar1,\bar3)
+\zt^r(2,\bar3,\bar1)+\zt^r(\bar3,\bar1,2)+\zt^r(\bar3,2,\bar1)=\\
\zt(\bar1)\zt(2)\zt(\bar3)-(1-2r)(\zt(\bar3)^2+\zt(\bar1)\zt(\bar5))
+2(1-3r+3r^2)\zt(6).
\end{multline*}
Eqs. (\ref{elem}-\ref{half}) also hold, provided we interpret $i\la_j$
in those formulas as the sum of $\la_j$ copies of $i\in\I$.

\end{document}